\documentclass[11pt]{amsart}
\usepackage{amssymb,amscd,amsthm,verbatim}
\usepackage{amsbsy}
\usepackage{latexsym}
\usepackage[all,cmtip]{xy}
\usepackage[mathscr]{euscript}

\usepackage{etex}
\usepackage{graphicx}
\usepackage{pictex}
\usepackage{epstopdf}
\usepackage{color}

\usepackage{bbm}

\date{\today}

\newcommand{\Z}{{\mathbb Z}}
\newcommand{\R}{{\mathbb R}}
\newcommand{\Q}{{\mathbb Q}}

\newcommand{\T}{{\mathbb T}}

\newcommand{\N}{{\mathbb N}}

\newcommand{\GL}{{\mathrm{GL}}}
\newcommand{\SL}{{\mathrm{SL}}}

\newcommand{\MM}{{\mathrm{M}}}

\newcommand{\Leb}{{\mathrm{Leb}}}

\newcommand{\Var}{{\widetilde{\mathrm{Var}}}}

\newtheorem{theorem}{Theorem}[section]
\newtheorem{lemma}[theorem]{Lemma}
\newtheorem{prop}[theorem]{Proposition}
\newtheorem{coro}[theorem]{Corollary}

\theoremstyle{definition}

\newtheorem{remark}[theorem]{Remark}

\theoremstyle{plain}

\allowdisplaybreaks \numberwithin{equation}{section}

\DeclareMathOperator{\supp}{supp}

\begin{document}

\title[Uniformity Aspects of Cocycles and Schr\"odinger Operators]{Uniformity Aspects of $\SL(2,\R)$ Cocycles and Applications to Schr\"odinger Operators Defined Over Boshernitzan Subshifts}

\author[D.\ Damanik]{David Damanik}
\address{Department of Mathematics, Rice University, Houston, TX~77005, USA}
\email{damanik@rice.edu}
\thanks{D.D.\ was supported in part by NSF grants DMS--1700131 and DMS--2054752, Simons Fellowship $\# 669836$, and an Alexander von Humboldt Foundation research award.}

\author[D.\ Lenz]{Daniel Lenz}
\address{Institute for Mathematics, Friedrich-Schiller University, Jena, 07743 Jena}
\email{daniel.lenz@uni-jena.de}

\begin{abstract}
We consider continuous $\SL(2,\R)$ valued cocycles over general
dynamical systems and discuss a variety of uniformity notions. In
particular, we provide a description of uniform one-parameter
families of continuous $\SL(2,\R)$ cocycles as $G_\delta$-sets.
These results are then applied to Schr\"odinger operators with
dynamically defined potentials.  In the case where the base dynamics
is given by a subshift satisfying the Boshernitzan condition, we
show that for a generic continuous sampling function, the associated
Schr\"odinger cocycles are uniform for all energies and, in the
aperiodic case,  the spectrum is a Cantor set of zero Lebesgue
measure.
\end{abstract}

\dedicatory{Dedicated to the memory of Michael Boshernitzan}

\maketitle

\section{Introduction}

Consider a compact metric space $\Omega$ and a homeomorphism $T : \Omega \to \Omega$. Such a pair $(\Omega, T)$ will be called a \emph{dynamical system} in this paper.\footnote{It would be more accurate to call it a topological dynamical system, but we hope this slight abuse of language does not lead to any confusion. Given that we are interested in topological notions and quantities, this is the natural setting for us to work in.} We will freely use standard concepts from the theory of dynamical systems such as minimality and unique ergodicity; see, for example, the textbook \cite{Wal}.

The set of real $2\times 2$ matrices with determinant equal to one is denoted by $\SL(2,\R)$. The elements of
$$
C(\Omega,\SL(2,\R) := \{ A : \Omega \to \SL(2,\R) : A \mbox{ continuous}\}
$$
are referred to as \emph{continuous} $\SL(2,\R)$ \emph{cocycles}.

Any continuous cocycle $A \in C(\Omega, \SL(2,\R))$ gives rise to the skew-product
$$
(T,A) : \Omega \times \R^2 \to \Omega \times \R^2, \; (\omega,v) \mapsto (T \omega, A(\omega)v).
$$
For $n \in \Z$, define $A_n : \Omega \to \SL(2,\R)$ by $(T,A)^n = (T^n,A_n)$.

A cocycle $A$  is called \emph{uniformly hyperbolic} if there exists $L > 0$ with
$$
\liminf_{n \to \infty} \frac{1}{n}\log\|A_n (\omega)\| \ge L
$$
uniformly in $\omega \in \Omega$.

One says that $A$ is \emph{uniform} if there is a number $L(A)$ such that
$$
\lim_{n \to \infty} \frac1n \log \|A_n(\cdot)\| = L(A)
$$
uniformly. Clearly, any uniform cocycle $A$  with $L(A)>0$ is uniformly hyperbolic.

A cocycle may or may not be uniform. However, by the subadditive ergodic theorem, once an ergodic measure $\mu$ is chosen, there is always a ($\mu$-dependent) $L_\mu(A)$ such that
$$
\lim_{n \to \infty} \frac1n \log \|A_n(\omega)\| = L_\mu(A) \text{ for $\mu$-almost every } \omega \in \Omega.
$$
The numbers $L(A)$ and $L_\mu(A)$ are called \emph{Lyapunov exponents}.

For our actual considerations, a further uniformity property of
cocycles will be relevant. A cocycle $A \in C(\Omega,\SL(2,\R))$ is
said to have \textit{uniform behavior} if it is either uniformly
hyperbolic,  or
$$
\limsup_{n \to \infty} \frac{1}{n} \log \|A_n (\omega)\| = 0
$$
uniformly in $\omega \in \Omega$. Note that this latter condition can also be written as  $\lim_{n \to \infty} \frac{1}{n} \log \|A_n (\omega)\| = 0$
uniformly in $\omega \in\Omega$ (as $\|B\|\geq 1$ for any $B \in \SL(2,\R)$, and hence $\log \|A_n (\omega)\| \geq 0$).

\begin{remark}\label{r.uniformvanishing}
It is known that the property
$$
\lim_{n \to \infty} \frac{1}{n} \log \|A_n (\omega)\| = 0 \text{ uniformly in } \omega \in\Omega
$$
is equivalent to the simultaneous vanishing of the Lyapunov exponent for all ergodic Borel probability measures $\mu$,
$$
\sup \{ L_\mu(A) : \mu \text{ ergodic} \} = 0;
$$
compare \cite[Proposition~1]{AB07}, \cite[Theorem~1]{S98}, and
\cite[Theorem~1.7]{SS00}. See also \cite{F97} for the special case
where there is only one ergodic measure and \cite{BN14} for related
work.
\end{remark}

Let us briefly discuss the relationship between these uniformity notions, see Appendix~\ref{a.UH} for more details. For uniquely ergodic dynamical systems, a continuous cocycle is uniformly hyperbolic if and only if it is uniform with $L(A) > 0$. From this we immediately conclude that for uniquely ergodic dynamical systems, a continuous cocycle is uniform if and only if it has uniform behavior. For general dynamical systems, it is obviously true that a uniform cocycle has uniform behavior. However, the converse does not hold. Indeed, for any non-uniquely ergodic system, there exist uniformly hyperbolic continuous cocycles that are not uniform.

We will be interested in one-parameter families of cocycles. This is
partly motivated by the application of our general results,
presented below, to the case of Schr\"odinger cocycles, which
naturally depend on the energy parameter. Let $I \subseteq \R$ be an
interval in $\R$ and equip $C(I\times \Omega, \SL(2,\R))$ with the topology of local uniform convergence. Define $ W(I, \Omega)$ to be the set
$$
\{ A \in C(I \times \Omega, \SL(2,\R)) : A(E,\cdot)\text{ has uniform behavior for each $E\in I$} \}.
$$
Then we have the following result:

\begin{theorem}\label{t:general-general}
Let $(\Omega,T)$ be a dynamical system and let $I \subseteq \R$ be
an interval. Then, $W(I,\Omega)$ is a $G_\delta$-set.
\end{theorem}

\begin{remark}\label{r:after-main-theorem}
(a) Of course the theorem can be applied with $I$ being just one
point. This gives that the set of $A \in C(\Omega,\SL(2,\R))$ with
uniform behavior is a $G_\delta$-set. This particular case was
known; see the first paragraph of the proof of
\cite[Theorem~1]{AB07}. In fact, under suitable assumptions on
$(\Omega,T)$, Avila and Bochi even show that it is a \emph{dense}
$G_\delta$-set \cite[Theorem~1]{AB07}.
\\[1mm]
(b) An inspection of the proof shows that $I$ could be chosen as any topological space that is a countable union of compact subspaces.
\end{remark}

As pointed out above, for uniquely ergodic dynamical systems, a
cocycle is uniform if and only if it has uniform behavior, but in
general the set of uniform cocycles may be strictly smaller than the
set of cocycles with uniform behavior. This naturally raises the
question whether the set of uniform cocycles is a $G_\delta$-set in
general. Thus, let us consider
$$
U(I, \Omega) := \{ A \in C(I \times \Omega, \SL(2,\R)) : A(E,\cdot) \text{ is uniform for each } E \in I \}.
$$

The following theorem answers the question affirmatively.

\begin{theorem}\label{t:main-abstract}
Let $(\Omega,T)$ be a dynamical system and let $I \subset \R$ be an
interval. Then, $U(I,\Omega)$ is a $G_\delta$-set.
\end{theorem}

\begin{remark}
(a) With the obvious modifications, parts (a) and (b) of Remark~\ref{r:after-main-theorem} apply here as well.
\\[1mm]
(b) For equicontinuous systems, it is known that the set of uniform
cocycles is a dense $G_\delta$-set (see Furman \cite{F97}).
\end{remark}

The results above are relevant in the study of spectral properties of discrete one-dimensional Schr\"odinger operators with dynamically defined potentials. Operators of this kind arise as follows. The set of continuous $f : \Omega\longrightarrow \R$ is denoted by $C(\Omega,\R)$. Any choice of an $f\in C(\Omega,\R)$, commonly referred to as a \emph{sampling function}, gives rise to \emph{potentials} $V_\omega(n) = f(T^n \omega)$, $\omega \in \Omega$, $n \in \Z$, and the associated \emph{Schr\"odinger operators}
$$
[H_\omega \psi] (n) = \psi(n+1) + \psi(n-1) + V_\omega(n) \psi(n)
$$
in $\ell^2(\Z)$. The spectral theory of such operators has been reviewed in \cite{D17} and it will be discussed in full detail in the forthcoming monographs \cite{DF21a, DF21b}. We refer to these works for details on the concepts and results discussed next.

If $\mu$ is a $T$-ergodic Borel probability measure, then the spectral properties of $H_\omega$ are $\mu$-almost surely independent of $\omega \in \Omega$. For example, there are sets $\Sigma, \Sigma_\mathrm{ac}, \Sigma_\mathrm{sc}, \Sigma_\mathrm{pp}$ such that $\sigma(H_\omega) = \Sigma$ and $\sigma_\bullet(H_\omega) = \Sigma_\bullet$, $\bullet \in \{ \mathrm{ac}, \mathrm{sc}, \mathrm{pp} \}$ for $\mu$-almost every $\omega \in \Omega$.

Several recent works have investigated the question of which
spectral properties are generic. One usually fixes the base dynamics
$(\Omega, T)$ and studies the set of $f \in C(\Omega,\R)$ for which
a certain spectral phenomenon occurs. For example, Avila and Damanik
showed in \cite{AD05} that $\{ f \in C(\Omega,\R) :
\Sigma_\mathrm{ac} = \emptyset \}$ is  a dense $G_\delta$-set for
any ergodic $\mu$, provided that $T$ is not periodic.\footnote{By
the standard theory of periodic Schr\"odinger operators, the result
clearly fails if the assumption is dropped.} A companion result was
obtained by Boshernitzan and Damanik in \cite{BD08}: $\{ f \in
C(\Omega,\R) : \Sigma_\mathrm{pp} = \emptyset \}$ is residual (i.e.,
it contains a dense $G_\delta$-set), provided that $(\Omega,T,\mu)$
has the metric repetition property. See \cite{BD08, BD09} for the
definition of this property and many examples, including shifts and
skew-shifts on tori.

The proofs of the results in \cite{AD05, BD08} just mentioned rely on approximation of $f$ by functions taking finitely many values. Realizing that the absence of point spectrum, as well as the absence of absolutely continuous spectrum, are phenomena that are quite well understood in the setting of sampling functions taking finitely many values, the results in \cite{AD05, BD08} then appear to be somewhat
natural.\footnote{It should be noted, however, that they were both initially quite surprising as one had previously expected the presence of absolutely continuous spectrum for small quasi-periodic potentials, and the presence of point spectrum for operators generated by the standard skew-shift $T(\omega_1,\omega_2) =
(\omega_1 + \alpha, \omega_1 + \omega_2)$.}


Let us discuss some key concepts underlying the general theory and the results just mentioned. A cocycle of the form
$$
\omega \mapsto \begin{pmatrix} g(\omega) & -1 \\ 1 & 0 \end{pmatrix}
$$
with $g \in C(\Omega,\R)$ is called a \emph{Schr\"odinger cocycle} and denoted by $A^g$. Given an operator family $\{ H_\omega \}_{\omega \in \Omega}$ as introduced above, the associated one-parameter family of Schr\"odinger cocycles $\{ A^{E - f} \}_{E \in \R}$ is intimately related to the study of the solutions of the associated difference equation
$$
u(n+1) + u(n-1) + V_\omega(n) u(n) = E u(n)
$$
and hence provides important information. The parameter $E$ is referred to as the \textit{energy} in this context.


We write
\begin{align*}
\mathcal{UH} & = \{ E \in \R : A^{E-f} \text{ is uniformly hyperbolic} \}, \\
\mathcal{Z} & = \{ E \in \R : L_\mu(A^{E-f}) = 0 \}, \\
\mathcal{NUH} & = \R \setminus (\mathcal{UH} \cup \mathcal{Z} ).
\end{align*}
Note that $\mathcal{Z}$ and $\mathcal{NUH}$ depend on the choice of ergodic measure $\mu$, while $\mathcal{UH}$ does not. This provides a ($\mu$-dependent) partition of the energy axis: $\R = \mathcal{UH} \sqcup \mathcal{NUH}  \sqcup \mathcal{Z} $.

Let us now relate the Lyapunov exponents with the spectra mentioned earlier. The Johnson-Lenz theorem \cite{J86, L02} states that
\begin{equation}\label{e.JL}
\mathcal{Z} \subseteq \Sigma \subseteq \mathcal{Z} \cup \mathcal{NUH}.
\end{equation}
Moreover,
\begin{equation}\label{e.JL2}
\supp \mu = \Omega \quad \Rightarrow \quad \Sigma = \mathcal{Z} \cup \mathcal{NUH}.
\end{equation}

Recall that the essential closure of a measurable set $M \subseteq \R$ is given by $\overline{M}^\mathrm{ess} = \{ E \in \R : \Leb(M
\cap (E - \varepsilon, E + \varepsilon) > 0 \text{ for every } \varepsilon > 0 \}$. The Ishii-Pastur-Kotani theorem
\cite{I73, K84, P80} (see also \cite{D07, K97} for an exposition) states that
\begin{equation}\label{e.IPK}
\Sigma_\mathrm{ac} = \overline{\mathcal{Z}}^\mathrm{ess}.
\end{equation}
Finally, if the potentials $\{ V_\omega \}$ take finitely many values and are $\mu$-almost surely aperiodic, then by Kotani \cite{K89}, we have
\begin{equation}\label{e.K}
\Leb(\mathcal{Z}) = 0,
\end{equation}
which by \eqref{e.IPK} implies that $\Sigma_\mathrm{ac} =  \emptyset$. The very general result \eqref{e.K} was alluded to in the discussion above as one of the general spectral phenomena in the setting of potentials taking finitely many values, and it forms the basis of the generic $C^0$ result from \cite{AD05} also mentioned above.

Note that under the assumption $\supp \mu = \Omega$ (which holds, e.g., when $T$ is minimal) \eqref{e.JL2} shows that $\Sigma = \mathcal{Z}$ if and only if $\mathcal{NUH} =\emptyset$. Now, for uniquely ergodic  dynamical systems uniform behavior is equivalent to uniformity, see appendix,  and $L_\mu(A) =0$ if and
only if $A$ is uniform with $L(A) =0$. Thus, for minimal uniquely ergodic dynamical systems we have
\begin{equation}\label{e.Zero-uniform}
\Sigma = \mathcal{Z} \Longleftrightarrow \mathcal{NUH}=\emptyset
\Longleftrightarrow A^{E-f} \text{ is uniform for all $E\in \R$}.
\end{equation}
For general systems it follows from the definitions and Remark~\ref{r.uniformvanishing} that
$$
\{ E \in \R : A^{E-f} \text{ has uniform behavior} \} = \mathcal{UH}
\cup \bigcap_{\mu \text{ ergodic}} \mathcal{Z}_\mu.$$
In other
words, uniform behavior fails for $A^{E-f}$ precisely when
$$
E\in \bigcup_{\mu \text{ ergodic}} \mathcal{NUH}_\mu.
$$
This gives
\begin{eqnarray*} \mathcal{NUH}_\mu = \emptyset  \text{
for each ergodic $\mu$ } &\Longleftrightarrow & A^{E-f} \text{ has
uniform behavior for all } E\in \R  \\
&\Longleftrightarrow & \Sigma_\mu = \mathcal{Z}_{\mathcal{U}}\text{
for each ergodic $\mu$}.
\end{eqnarray*}
Here, $$\mathcal{Z}_\mathcal{U}:=\{ E\in \R : A^{E-f} \text{
is uniform with } L(A^{E-f}) = 0\} =  \bigcap_{\mu \text{ ergodic}}
\mathcal{Z}_\mu.
$$

In any case, if the potentials $\{ V_\omega \}$ take finitely many
values, then \eqref{e.K} implies zero-measure spectrum whenever one
can show that $\mathcal{NUH} = \emptyset$. Thus, pursuing a proof of
the absence of non-uniformity is a natural approach to zero-measure
spectrum whenever a property such as \eqref{e.K} is known. This
approach is implemented in \cite{DL06a, L02}, as well as in the
present paper.

Let us mention that the zero-measure spectrum property has been investigated extensively for sampling functions taking finitely many values. From the classical results for the Fibonacci Hamiltonian \cite{S89} or the more general class of operators with Sturmian potentials \cite{BIST89} through numerous results for operators with potentials generated by substitutions to the general result
\cite{DL06a} by Damanik and Lenz, which covers many examples \cite{DL06b}, this is a spectral statement that is quite ubiquitous in this setting.

It has therefore been a very natural open problem to find conditions on the base dynamics $T : \Omega \to \Omega$ such that $\{ f \in C(\Omega,\R) : \Leb(\Sigma) = 0 \}$ is residual. The paper
\cite{ADZ14} by Avila, Damanik, and Zhang discusses this question in
the particular case $T : \R/\Z \to \R/\Z$, $\omega \mapsto \omega +
\alpha$, $\alpha \not\in \Q$, but fails to answer it. Instead,
\cite{ADZ14} proves the weaker result that the singularity of the
integrated density of states is generic in this setting.

Not only is the problem open in the case of irrational circle rotations, it is open in \emph{any} setting and hence one of our goals is to exhibit the first class of base dynamics $T : \Omega \to \Omega$ for which $\{ f \in C(\Omega,\R) : \Leb(\Sigma) = 0 \}$ is a dense $G_\delta$-set. At the same time we  will provide the
first class of aperiodic base dynamics for which $\{ f \in C(\Omega,\R) : \mathcal{NUH} = \emptyset\}$ or, equivalently, $\{f\in C(\Omega,\R) : \Sigma = \mathcal{Z}\}$ is a dense $G_\delta$-set. This is of interest as the equality $\Sigma = \mathcal{Z}$ is known in the periodic case and, hence, aperiodic dynamics giving this feature deserve particular attention.

We will work with aperiodic subshifts that satisfy the Boshernitzan
condition. Recall that a \emph{subshift} is a closed shift-invariant
subset $\Omega$ of $A^\Z$, where $A$ is a finite set carrying the
discrete topology and $A^\Z$ is endowed with the product topology.
The map $T : \Omega \to \Omega$ is given by the shift $(T \omega)_n
= \omega_{n+1}$, and it is clearly a homeomorphism. We say that a
subshift $\Omega$ satisfies the \emph{Boshernitzan condition} (B) if
it is minimal and there is a $T$-invariant Borel probability measure
$\mu$ such that
$$
\limsup_{n \to \infty} n \cdot \min \{ \mu([w]) : w \in \Omega_n \} > 0.
$$
Here $\Omega_n = \{ \omega_1 \ldots \omega_n : \omega \in \Omega \}$ is the set of words of length $n$ that occur in elements of $\Omega$ and $[w]$ is the cylinder set $[w] = \{ \omega \in \Omega : \omega_1 \ldots \omega_n = w \}$. This condition was introduced by Boshernitzan in \cite{B92} as a sufficient condition for unique ergodicity.

\begin{theorem}\label{t.main}
Suppose $\Omega$ is a subshift that satisfies the Boshernitzan
condition $\mathrm{(B)}$. Then, the following holds:

{\rm (a)} The set
$$
\{ f \in C(\Omega,\R) : \mathcal{NUH} = \emptyset \} = \{f\in C(\Omega,\R) : \Sigma = \mathcal{Z}\}
$$
is a dense $G_\delta$-set.

{\rm (b)} If $\Omega$ is furthermore aperiodic, then the set
$$
\{f\in C(\Omega,\R) : \Leb(\Sigma) =0\}
$$
is a dense $G_\delta$-set.
\end{theorem}


The theorem has the following immediate consequence.

\begin{coro}\label{c.main}
Suppose $\Omega$ is an aperiodic subshift that satisfies the
Boshernitzan condition $\mathrm{(B)}$. Then, zero-measure spectrum
given by the vanishing set of the Lyapunov exponent  is generic,
that is, $$\{ f \in C(\Omega,\R) : \Leb(\Sigma) = 0 \mbox{ and }
\Sigma = \mathcal{Z} \}$$
 is a dense $G_\delta$-set.
\end{coro}

\begin{remark}\label{r.afterCorollary1.7}
(a) It is well known that the spectrum is always closed and, in the
dynamically defined setting we consider, it never contains any
isolated points. Thus, Corollary~\ref{c.main} shows that Cantor
spectrum of zero Lebesgue measure is generic when the base dynamics
is given by an aperiodic subshift that satisfies $\mathrm{(B)}$.
\\[1mm]
(b) As pointed out above, if the subshift $\Omega$ satisfies (B),
then it is uniquely ergodic by \cite[Theorem~1.2]{B92}. For this
reason there is no ambiguity in writing $\Sigma$ without specifying
$\mu$. On the other hand, the minimality of $\Omega$ and the
continuity of the sampling functions $f$ in question also imply the
independence of the spectrum of $\omega$, so that in the setting of
Theorem~\ref{t.main}, $\sigma(H_\omega) = \Sigma$ for every $\omega
\in \Omega$, not merely for $\mu$-almost every $\omega \in \Omega$.
\\[1mm]
(c) Many important classes of subshifts satisfy (B); see
\cite{DL06b} for a detailed discussion.
\\[1mm]
(d) It remains very interesting to clarify whether zero-measure
spectrum is ($C^0$-) generic for quasi-periodic potentials, or at
least for one-frequency quasi-periodic potentials.
\end{remark}

Finally, we note that our general result, Theorem~\ref{t:main-abstract}, can also be seen in the context of a question of Walters on existence of non-uniform cocycles.
Specifically, Walters asks in \cite{Wal1} whether every uniquely ergodic dynamical system (with non-atomic invariant measure) allows for a non-uniform cocycle. Walters discusses some examples,  where the answer is affirmative.  The question in general  seems to still be open with further partial results contained in \cite{F97}. In
this situation, the following consequence of (the proof of) our spectral results may be of interest.

\begin{coro}\label{c:walters}
Suppose $\Omega$ is an aperiodic subshift that satisfies the Boshernitzan condition $\mathrm{(B)}$. Then, the set of uniform cocycles is a dense $G_\delta$-set.
\end{coro}

\begin{remark}
Based on these considerations we feel that aperiodic Boshernitzan subshifts are the best candidates for a potential negative answer to Walters' question, but at this time we are unable to extend the uniformity result to all continuous cocycles over a Boshernitzan subshift.
\end{remark}

The paper is organized as follows. We prove
Theorem~\ref{t:general-general} in Section~\ref{s.3} and
Theorem~\ref{t:main-abstract} in Section~\ref{s.4}. We then provide
a result on semicontinuity of the measure of the spectrum for
general dynamical systems in Section~\ref{s.6} and a result on
denseness of cocycles for subshifts in Section \ref{s.7}.  In
Section~\ref{s.5} we then  derive Theorem~\ref{t.main} from results
in the earlier sections. That section contains also the  proof of
Corollary~\ref{c:walters}.  Finally, there are two appendices, one
discussing the relationships between the uniformity notions we
consider, and one discussing a consequence of the Avalanche
Principle that we need in the earlier sections.

\section*{Acknowledgment}

Our original version of part (b) of  Theorem~\ref{t.main} was based on Theorem~\ref{t.main2}. We are indebted to Jake Fillman for pointing out that the ideas in \cite{DFL17} should make the proof possible that we now present.



\section{Cocycles With Uniform Behavior as a $G_\delta$-Set}\label{s.3}

In this section we prove Theorem~\ref{t:general-general}. That is, we show that the set of cocycles with uniform behavior is a $G_\delta$-set, and in fact we prove this result for families of cocycles depending on one real parameter.

We start with  a simple observation.

\begin{lemma}\label{l:aux-zwei}
Let $(\Omega,T)$ be a dynamical system and $A \in C(\Omega, \SL(2,\R))$. If there exist $L > 0$ and $k \in \N$ with $M := \max \{ \frac{1}{k} \log \|A_k(\omega)\| : \omega \in \Omega\} < L$, then
$$
\frac{1}{n} \log\|A_n (\omega)\| < L
$$
for all $\omega \in \Omega$ and
$$
n \geq \frac{2 k \max \{ \log \|A(\omega)\| : \omega\in \Omega\}}{L- M}.
$$
\end{lemma}

\begin{proof}
Set  $N := \frac{2  k  \max \{ \log \|A(\omega)\| : \omega \in \Omega\}}{L- M}$. By definition of $N$, we have
\begin{equation}\label{hilfe}
\frac{1}{N} \log \|A_r(\omega)\| \leq  \frac{L-M}{2}
\end{equation}
for all $\omega \in \Omega$, $r = 0, \ldots, k$. Clearly, this estimate continues to hold if $N$ is replaced by any $n\geq N$.

Consider now an $n \in \N$ with $n \geq N$. Of course,  $n$ can be uniquely written in the form $n = s k + r$ with $s \in \N\cup \{0\}$ and $0 \leq r < k$. By construction of the cocycle, we obtain
$$
A_{n}(\omega) = A_r (T^{s k} \omega) A_k (T^{(s-1)k} \omega) \ldots A_k (T^k\omega) A_k(\omega).
$$
Taking logarithms and using submultiplicativity of the matrix norm and additivity of the logarithm, we find
\begin{eqnarray*} \frac{1}{n}\log\|A_n (\omega)\| &\leq &
\frac{1}{n} \log\|A_r (T^{sk}\omega)\| + \frac{1}{n}\sum_{j=0}^{s-1}
\log\|A_k (T^{jk} \omega)\|\\
\eqref{hilfe} &\leq& \frac{L-M}{2} + \frac{1}{n}\sum_{j=0}^{s-1} k
\frac{\log\|A_k (T^{jk} \omega)\|}{k}\\
(\mbox{definition of $M$}) &\leq & \frac{L-M}{2} +
\frac{1}{n}\sum_{j=0}^{s-1} k M\\
(sk \leq n) &=& \frac{L-M}{2} + M \\
(M<L) &< & L
\end{eqnarray*}
This finishes the proof.
\end{proof}

\begin{proof}[Proof of Theorem \ref{t:general-general}]
We first consider a compact interval $I$.

For $\varepsilon > 0$, we define $W_\varepsilon$ to be the set of $A \in C(I \times \Omega, \SL(2,\R))$ such that for each $E \in I$, the cocycle $A(E,\cdot)$ is uniformly hyperbolic or there exists a $k \in \N$ with $\frac{1}{n} \log \|A_n (\omega)\| < \varepsilon$ for all $\omega \in \Omega$ and $n \geq k$. Clearly,
$$
W = \bigcap_{m\in\N} W_{\frac{1}{m}}.
$$
Thus it suffices to show that $W_\varepsilon$ is open for any $\varepsilon > 0$. To do so, we consider $E \in I$ arbitrary. There are two cases:

\smallskip

\textit{Case 1: $A(E,\cdot)$ is uniformly hyperbolic.} As is well-known, the set of uniformly hyperbolic cocycles is open (see, e.g., \cite{Z19}). As $A$ is continuous in the first variable, there exists a $\delta > 0$ such that any $B \in C(I \times \Omega, \SL(2,\R))$ close enough to $A$ will have the property that $B(E',\cdot)$ is uniformly hyperbolic for all $E' \in (E-\delta,E+\delta) \cap I$.

\smallskip

\textit{Case 2:  $A(E,\cdot)$ satisfies $\frac{1}{k} \log \|A_k (E,\omega)\| < \varepsilon$ for all $\omega \in \Omega$ for some $k \in \N$.} By continuity of $A$ and compactness of $\Omega$, there exists a $\delta > 0$ with
$$
\sup_{\omega \in \Omega, E' \in (E-\delta, E+\delta) \cap I} \frac{1}{k} \log \|A_k (E',\omega)\| < \varepsilon.
$$
This same inequality will then also hold for any $B \in C(I \times \Omega, \SL(2,\R))$ sufficiently close to $A$. By Lemma~\ref{l:aux-zwei} there exists then an $N \in \N$ with
$$
\frac{1}{n} \log \|B_n (E',\omega)\| < \varepsilon
$$
for all $\omega \in \Omega$, $n \geq N$ and $E'\in (E - \delta , E + \delta) \cap I$ for all such $B$.

\smallskip

So, in both  of these two cases there is an open neighborhood $(E - \delta, E + \delta) \cap I$ of $E$ such that any $B$ sufficiently close to $A$ shares the respective property of $A(E,\cdot)$ for all $E'$ in this neighborhood. As $I$ is compact, the openness of $W_\varepsilon$ then follows by standard reasoning.

\smallskip

We now consider an arbitrary interval $I$ in $\R$. We can write $I$ as a countable union of compact intervals $I_n$, i.e.\ $I = \bigcup_{n \in \N}  I_n$. By what we have shown already, $W(I_n,\Omega)$ is a $G_\delta$-set for each $n \in \N$. For any $n \in \N$, there is the canonical embedding  $j_n : I_n \times \Omega \longrightarrow I \times \Omega, (E,\omega) \mapsto (E,\omega)$, and the associated restriction map
$$
R_n : C(I\times \Omega,\SL(2,\R)) \longrightarrow C(I_n \times \Omega, \SL(2,\R)), \; A \mapsto A \circ j_n.
$$
Then, $R_n$ is continuous. Hence, $R_n^{-1} (W(I_n,\Omega))$ is a $G_\delta$-set for each $n \in \N$  (as the inverse image of a $G_\delta$-set under a continuous map) and so is then
$$
W(I,\Omega) = \bigcap_{n \in \N} R_n^{-1} (W(I_n,\Omega)).
$$
This finishes the proof.
\end{proof}

\section{Uniform Cocycles are a $G_\delta$-Set}\label{s.4}

In this section we prove Theorem~\ref{t:main-abstract}. A pertinent
idea is that for a uniquely ergodic dynamical system $(\Omega,T)$, a
continuous $B : \Omega \to \SL(2,\R)$ is uniform if and only if
\begin{equation}\label{e.uniformitycondition}
\lim_{n\to \infty} \frac{1}{n} \sup_{\omega, \varrho \in \Omega} \left| \log \|B_n (\omega)\| - \log \|B_n (\varrho)\| \right| = 0.
\end{equation}

Indeed, it is clear that any uniform $B$ will satisfy \eqref{e.uniformitycondition}. Conversely, any $B$ satisfying \eqref{e.uniformitycondition} must be uniform as there exists (by the subadditive ergodic theorem) an $\omega_0 \in \Omega$ with $\lim_{n \to \infty} \frac{1}{n} \log \|B_n (\omega_0)\| = L(B)$. Some additional work will be needed to deal with the dependence on the parameter.

We start with two  auxiliary statements. For the convenience of the
reader we include sketches of the proofs.

\begin{lemma}\label{l:aux}
Let $(\Omega,T)$ be a dynamical system and $A \in C(\Omega,\SL(2,\R)$ be arbitrary.
\begin{itemize}

\item[(a)] If there exist $L > 0$ and $N \in \N$ with $\frac{1}{k} \log \|A_k(\omega)\| < L$ for all $\omega \in \Omega$ and $k = N,\ldots, 2N$, then
$$
\frac{1}{n} \log\|A_n (\omega)\| < L
$$
for all $\omega \in \Omega$ and $n \geq N$.

\item[(b)] Let $c := \max_{\omega \in \Omega} \{ \log \|A(\omega)\|, \log \|A^{-1}(\omega)\| \}$. Then, for any $n \in \N$,
$$
\left| \frac{\log \|A_{n+1}(\omega)\|}{n+1} - \frac{\log\|A_n (\omega)\|}{n}\right| \leq \frac{1}{n+1} \frac{\log\|A_n (\omega)\|}{n} + \frac{c}{n+1}.
$$
\end{itemize}
\end{lemma}

\begin{proof}
(a) Consider $n \geq N$. Then, we can uniquely write $n$ in the form $n = k N + r$ with $k \in \N \cup \{0\}$ and $N \leq r \leq 2 N-1$. Now, the proof follows similar lines as the proof of Lemma~\ref{l:aux-zwei}.

\smallskip

(b) For invertible matrices $C,B$, we clearly have $\|B C\| \leq \|B\| \|C\|$ and $\|C\| = \|B^{-1} B C\| \leq \|B^{-1}\| \|B C\|$. Applying this with $C = A_n (\omega)$ and $B = A(T^n \omega)$, we infer (b) after a short computation.
\end{proof}

\begin{remark}
It follows from part (a) of the lemma that for any $L > 0$ and $N \in \N$, the set of $A \in C(\Omega,\SL(2,\R))$ with $\sup_{\omega \in \Omega, n \geq N} \frac{1}{n} \log\|A_n (\omega)\| < L$ is open.
\end{remark}

\smallskip

We now  show that the pointwise uniformity of the $A(E,\cdot)$ appearing in the definition of $U(I,\Omega)$ can be replaced by a uniform uniformity when $I$ is compact. This is the content of the next proposition.

\begin{prop}\label{p:inclusion-one}
Let $(\Omega,T)$ be a  dynamical system. Let $I \subset \R$ be a compact interval. Consider $A \in U(I, \Omega)$. Then, for any $\varepsilon > 0$, there exists $N \in \N$ with
$$
\frac{1}{n}| \log \|A_n(E,\omega)\| - \log \|A_n (E,\varrho)\| | < \varepsilon
$$
for all $\omega, \varrho \in \Omega$, $E \in I$, and $n \geq N$.
\end{prop}

\begin{proof}

As $I$ is compact, it suffices to find for each $E \in I$, a $\delta > 0$ such that the desired estimate holds in $(E - \delta, E + \delta) \cap I$. We consider two cases:

\smallskip

\textit{Case 1: $A(E,\cdot)$ is uniform with $L(A(E,\cdot)) > 0$}. The proof follows from Lemma~\ref{l:consequence-avalanche} in the following way: Assume without loss of generality $\frac{\varepsilon}{46 L} < \frac{1}{12}$. By uniformity of $A(E,\cdot)$, there exists $N \in \N$ such that the assumptions of Lemma~\ref{l:consequence-avalanche} will be satisfied with $L = L(A(E,\cdot))$, $\ell = N$ and $\frac{\varepsilon}{ 47 L}< \frac{1}{12}$ instead of $\varepsilon$. Now, as discussed in part (c) of Remark~\ref{r:nach-avalanche-lemma}, the assumptions are open assumptions in the following sense: If they are satisfied for the cocycle $A(E,\cdot)$ with $\frac{\varepsilon}{ 47   L}$, then for any $\frac{1}{12} > \varepsilon' > \frac{\varepsilon}{47 L}$, any $B$ sufficiently close to $A(E,\cdot)$ will satisfy the assumptions as well with $\varepsilon'$ instead of $\frac{\varepsilon}{47 L}$ and the same $L$ and $\ell$. So, the conclusion of the lemma will hold for such $B$. With $\varepsilon' = \frac{\varepsilon }{46 L}$, the conclusion of the lemma gives
$$
L \Big( 1 - \frac{44}{46 L}\varepsilon \Big) \leq \frac{1}{n} \log\|B_n (\omega)\| \leq L \Big( 1 + \frac{\varepsilon}{46 L} \Big)
$$
for all $n \geq \ell$ and $\omega \in \Omega$ for any such $B$. This in turn implies
$$
\frac{1}{n}| \log \|B_n(\omega)\| - \log \|B_n (\varrho)\| | < \varepsilon
$$
for all $\omega \in \Omega$ and $n \geq \ell = N$ for any such $B$. By continuity of $A$ (in the first variable), there exists $\delta > 0$ such that each $A(E',\cdot)$ with $E' \in (E - \delta, E + \delta) \cap I$ is such a $B$. This gives the desired statement.

\smallskip

\textit{Case 2:  $A(E,\cdot)$ is uniform with $L(A(E,\cdot))=0$.} In this case, there exists $N \in \N$ with
$$
\frac{1}{k} \log \|A_k (E,\omega)\| < \varepsilon / 3
$$
for all $\omega \in \Omega$ and $k \geq N$. By continuity of $A$, there exists a $\delta > 0$ with
$$
\frac{1}{k} \log \|A_k (E',\omega)\| < \varepsilon / 2
$$
for all $E' \in (E - \delta, E + \delta) \cap I$,  $\omega \in \Omega$ and $k =N,\ldots, 2N$. By (a) of Lemma~\ref{l:aux}, we find
$$
\frac{1}{n} \log \|A_n(E',\omega)\| < \varepsilon / 2
$$
for all $\omega \in \Omega$, $E' \in (E - \delta, E + \delta) \cap I$, and $n \geq N$, and this easily gives the desired statement in this case.
\end{proof}

Whenever $(\Omega,T)$ is a dynamical system and $I$ is a compact interval, we define for $n \in \N$,
$$
\Var_n : C(I \times \Omega, \SL(2,\R)) \longrightarrow [0,\infty)
$$
by
$$
\Var_n (A) := \sup_{E \in I} \sup_{\varrho, \omega \in \Omega} \{ \left| \log \|A_n(E,\omega)\| - \log \|A_n(E,\varrho)\| \right| \}.
$$

By the preceding proposition, any $A \in U(I,\Omega)$ satisfies
$$
\lim_{n \to \infty} \frac{1}{n} \Var_n (A) = 0.
$$
In fact, also an even stronger converse holds. This is the content of the next proposition.

\begin{prop}\label{p:inclusion-two}
Let $(\Omega,T)$ be a dynamical system. Let $I \subset \R$ be a compact interval. Then, any $A \in C(I \times \Omega, \SL(2,\R))$ with
$$
\liminf_{n \to \infty} \frac{1}{n} \Var_n (A) = 0
$$
belongs to $U(I,\Omega)$.
\end{prop}

\begin{proof}
Choose $E \in I$ arbitrary and write $A$ instead of $A(E,\cdot)$. By the assumption on $A$, we can find $n_k \in \N$ with
\begin{equation}
\tag{*} \delta_k := \frac{1}{n_k} \Var_{n_k} (A) \to 0, \; k \to \infty.
\end{equation}

We consider two cases:

\smallskip

\textit{Case 1: There exists $\omega_0\in\Omega$ with $\liminf_{k \to \infty} \frac{1}{n_k} \log \|A_{n_k} (\omega_0)\| = 0$.} Without loss of generality we can assume $\lim_{k \to \infty} \frac{1}{n_k} \log\|A_{n_k} (\omega)\| = 0$. By $(*)$ this gives $\lim_{k \to \infty} \frac{1}{n_k} \log \|A_{n_k} (\omega)\| = 0$ uniformly in $\omega \in \Omega$. By  Lemma~\ref{l:aux-zwei}, we infer $\lim_{n \to \infty} \frac{1}{n} \log \|A_n (\omega)\| = 0$ uniformly.

\smallskip

\textit{Case 2:  There exists $\omega_0 \in \Omega$ with $L := \liminf_{k \to \infty} \frac{1}{n_k} \log \|A_{n_k} (\omega_0)\| > 0 $.} Assume without loss of generality that
$$
L_k := \frac{1}{n_k} \log\|A_{n_k} (\omega_0)\| \to L, \; k \to \infty.
$$

By $(*)$ and the definition of $\Var_n$, we then have for all $\omega \in \Omega$,
\begin{equation}\label{twostar}
\tag{**} L_k - \delta_k \leq \frac{1}{n_k} \log\|A_{n_k}\| (\omega)
< L_k + \delta_k
\end{equation}
with $\delta_k \to 0$, $k \to \infty$, and $L_k \to L$, $k \to \infty$. By (b) of Lemma~\ref{l:aux}, we can assume without loss of generality that each $n_k$ is even (as we could otherwise replace $n_k$ by $n_k +1 $).

By $n_k \to \infty$,   Lemma~\ref{l:aux-zwei} and the upper bound in \eqref{twostar}, there exist $\delta_k' > 0$ with $\delta_k' \to 0$, $k \to \infty$ and
$$
\frac{1}{n} \log \|A_{n} (\omega)\| \leq L + \delta_k'
$$
for all $n \geq n_k/2$. Also, by $n_k \to \infty$, we clearly have
$\frac{3}{4} L \frac{n_k}{2} \geq \lambda_0$ (with $\lambda_0$ from
Lemma~\ref{l:consequence-avalanche}) for all sufficiently large $k$.

From these considerations we see that for arbitrary $\varepsilon <
\frac{1}{12}$, the assumptions of
Lemma~\ref{l:consequence-avalanche} are satisfied with $\ell =
\frac{n_k}{2}$, provided that $k$ is sufficiently large. The
statement of the lemma then gives the desired uniformity of $A$.
\end{proof}

\begin{proof}[Proof of Theorem~\ref{t:main-abstract}]
It suffices to consider a compact interval $I$ (compare the proof of Theorem~\ref{t:general-general}). Set
$$
U_{n,\varepsilon} := \Big\{ A \in C(I \times \Omega,\SL(2,\R)) : \frac{1}{n} \Var_n (A) < \varepsilon \Big\}.
$$
By continuity of $A$, the set $U_{n,\varepsilon}$ is open. Hence,
$$
\widetilde{U}_{N,\varepsilon} := \bigcup_{n \geq N} U_{n,\varepsilon}
$$
is open as well. Thus,
$$
W := \bigcap_{N, k \in \N} \widetilde{U}_{N, \frac{1}{k}}
$$
is a $G_\delta$-set.

\smallskip

It remains to show $W = U(I,\Omega)$. To show this, we prove two inclusions:

\smallskip

$U(I,\Omega) \subset W$: This is a direct consequence of Proposition~\ref{p:inclusion-one}.

\smallskip

$W \subset U(I,\Omega)$: It is not hard to see that
$$
W = \Big\{  A \in C(I \times \Omega,\SL(2,\R)) : \liminf_{n \to \infty} \frac{1}{n} \Var_n (A) = 0 \Big\}.
$$
Now, the inclusion follows from Proposition~\ref{p:inclusion-two}.
\end{proof}

\section{Upper Semicontinuity of the Measure of the Spectrum}\label{s.6}

In this section we consider a dynamical system $(\Omega,T)$ and the
associated Schr\"odinger operators and note that the map
\begin{equation}\label{e.msigmadef}
M_\Sigma: C(\Omega,\R) \to [0,\infty), \quad f \mapsto
\Leb(\Sigma_f)
\end{equation}
is upper semi-continuous.  The proof uses variations of ideas
developed in \cite{DFL17} in the context of continuum limit-periodic
Schr\"odinger operators and was suggested to us by Jake Fillman.
This will then imply that $\{ f \in C(\Omega,\R) : \Leb(\Sigma_f) =
0 \}$ is a $G_\delta$-set.

\begin{prop}\label{p.semicont}
The map $M_\Sigma$ defined in \eqref{e.msigmadef} is upper
semi-continuous, that is, for every $\delta > 0$, we have that
$M_\Sigma(\delta) := \{ f \in C(\Omega,\R) : \Leb(\Sigma_f) < \delta
\}$ is open.
\end{prop}

\begin{proof}
Let $\delta > 0$ be given, and let us consider $f \in
M_\Sigma(\delta)$. We have to show that there exists $\varepsilon >
0$ such that every $g \in C(\Omega,\R)$ with $\|f - g\|_\infty <
\varepsilon$ belongs to $M_\Sigma(\delta)$ as well.

By assumption, we have $\varepsilon' := \delta - \Leb(\Sigma_f) >
0$. By basic properties of the Lebesgue measure, we can choose
finitely many open intervals $I_1, \ldots, I_m$ with
$$
\Sigma_f \subset \bigcup_{j = 1}^m I_j \quad \text{and} \quad
\sum_{j = 1}^m |I_j| < \Leb(\Sigma_f) + \frac{\varepsilon'}{2}.
$$

Let us set $\varepsilon := \frac{\varepsilon'}{4m} > 0$. By
well-known properties of the spectrum of a Schr\"odinger operator
with respect to $\ell^\infty$ perturbations of the potential, if
$\|f - g\|_\infty < \varepsilon$, then $\Sigma_g \subset
B_\varepsilon(\Sigma_f)$ (where the latter notation denotes the
$\varepsilon$ neighborhood).

Putting these two ingredients together, we obtain
$$
\Sigma_g \subset B_\varepsilon \left( \bigcup_{j = 1}^m I_j \right),
$$
and hence
$$
\Leb(\Sigma_g) \le \Leb \left( B_\varepsilon \left( \bigcup_{j =
1}^m I_j \right) \right) \le 2m \varepsilon + \sum_{j = 1}^m |I_j| <
2m \varepsilon + \Leb(\Sigma_f) + \frac{\varepsilon'}{2} = \delta,
$$
as desired. This completes the proof.
\end{proof}

\begin{remark} The statement of the proposition  can also be
understood as follows: Let $\mathcal{K}$ be  set of all compact
subsets of $\R$ equipped with the the Hausdorff metric $d_H$ and let
$S(\ell^2 (\Z))$ be the set of bounded self-adjoint operators
equipped with the operator norm $\|\cdot\|$. Then, the map
$S(\ell^2(\Z))\ni A\mapsto \sigma(A)\in \mathcal{K}$, mapping a
bounded self-adjoint operator to its spectrum is continuous and,
actually, satisfies  $d_H(\sigma(A),\sigma(B))\leq \|A-B\|$, by
well-known perturbation theory of self-adjoint operators. Moreover,
the map $\mathcal{K}\ni K\mapsto \Leb(K)\in [0,\infty)$ is upper
semi-continuous, as is certainly well-known (and can also be seen
from the proof above). Altogether, we find that the map $S(\ell^2
(\Z))\longrightarrow [0,\infty), A\mapsto \Leb (\sigma(A)),$ is
upper semi-continuous. The statement of the proposition then follows
by composition as the map $C(\Omega,\R)\longrightarrow
S(\ell^2(\Z))$, $f\mapsto H_\omega^f$, is continuous with
$\|H_\omega^f - H_\omega^g\|\leq \|f-g\|_\infty$ for each
$\omega\in\Omega$.
\end{remark}

\begin{coro}\label{c.gdeltaset} Let $(\Omega,T)$ be a dynamical
system. Then, the set $\{ f \in C(\Omega,\R) : \Leb(\Sigma_f) = 0
\}$ is a $G_\delta$-set.
\end{coro}

\begin{proof}
Simply write
$$
\{ f \in C(\Omega,\R) : \Leb(\Sigma_f) = 0 \} = \bigcap_{n \in \N}
M_\Sigma \left( \tfrac1n \right)
$$
and use the fact that each $M_\Sigma(1/n)$ is open by
Proposition~\ref{p.semicont}.
\end{proof}

\section{Denseness of Locally Constant Cocycles}\label{s.7}
In this section we consider subshifts. Clearly, the set of locally
constant cocycles is dense in the set of continuous cocycles over a
subshift. Here, we show that a similar result holds for
one-parameter families of cocycles.

\begin{lemma}
Let $(\Omega,T)$ be a subshift and let $I$ be an interval
in $\R$. Then, the set
$$
\{A\in C(I \times \Omega,\SL(2,\R)) : A(E,\cdot) \mbox{ is locally
constant for each $E\in I$}\}
$$
is dense in $C(I\times \Omega, \SL(2,\R))$.
\end{lemma}

\begin{proof}
We consider $\SL(2,\R)$ as a subspace of the space $\MM(2,\R)$ of real $2\times 2$-matrices with metric induced by the standard norm on these matrices.

Let $A : I\times \Omega \longrightarrow \SL(2,\R)$ continuous,
$\varepsilon >0$, and $J\subset I$ compact be given.

We will construct a continuous $A' : J\times \Omega\longrightarrow
\SL(2,\R)$ such that $A'(E,\cdot)$ is locally constant for each
$E\in J$ and
$$
\|A(E,\omega)- A'(E,\omega)\|\leq  \varepsilon
$$
holds for all $E\in J$ and $\omega \in \Omega$.  This $A'$ can then be extended to a continuous function $A^* : I\times \Omega\longrightarrow \SL(2,\R)$, which is locally constant in the second argument,  by extending it constantly outside of the compact $J$. Specifically, with $J=[E_{\min},E_{\max}]$  we define $A^*(E,\omega):= A'(E_{\max},\omega)$ for $E\geq E_{\max}$ and $A^*(E,\omega) = A'(E_{\min},\omega)$ for $E\leq E_{\min}$.

As $A$ is continuous,  the set  $A(J\times \Omega)\subset \MM(2,\R)$
is compact. Hence, as the determinant is a continuous function on
$\MM(2,\R)$ and
$$
\det C = 1 \quad \text{and} \quad  \frac{1}{\sqrt{\det C}} C - C = 0
$$
for any $C \in A(J\times \Omega)$ (since $A(J\times \Omega) \subseteq \SL(2,\R)$), there exists a $\delta>0$ such that
$$
\det C > 0 \quad \text{and} \quad \left\| \frac{1}{\sqrt{\det C}}  C - C \right\| \leq \frac{\varepsilon}{2}
$$
for any $C\in \MM(2,\R)$ with distance from $A(J\times \Omega)$
smaller than $\delta$. Without loss of generality we assume $\delta
\leq \frac{\varepsilon}{2}$.

By continuity of $A$ again, we can find finitely many open sets $I_1,\ldots, I_N$ in $I$ with $$J\subset \bigcup_k I_k$$ such that
$$
\|A(E,\omega) - A(E',\omega)\|< \frac{\delta}{2}
$$
for all $\omega\in \Omega$ whenever $E',E$ belong to the same $I_k$. As locally constant cocycles are dense in $C(\Omega,\SL(2,\R))$ we can then choose for each  $k=1,\ldots, N$ a locally constant $B_k \in C(\Omega,\SL(2,\R)$ with
$$
\|B_k(\omega)- A(E,\omega)\| < \delta
$$
for any $\omega\in \Omega$ and $E\in I_k$.

Let $\varphi_k$, $k=1,\ldots, N$,   be a partition of
unity subordinate to $I_1,\ldots, I_N$.  This means that each
$\varphi_k$ is a continuous non-negative function on $I$  with compact
support contained in $I_k$  and
$$\sum_{k} \varphi_k (E) =1$$
 for each $E\in J$.  Define
$$A_k :J\times \Omega\longrightarrow \MM(2,\R), (E,\omega)\mapsto
\varphi_k (E) B_k (\omega).$$ Then, each $A_k$ is a continuous
function and $A_k (E,\cdot)$ is locally constant for each $E\in J$. Hence,
$$
\widetilde{A}:=\sum_k A_k
$$
is a continuous function on $J\times \Omega$ and $\widetilde{A}(E,\cdot)$ is locally constant for each $E\in J$. A short computation invoking $A(E,\omega) = \sum_k \varphi_k(E) A(E,\omega)$ for all $E\in J$ and $\omega\in\Omega$ shows
$$
\|\widetilde{A}(E,\omega) - A(E,\omega)\| \leq \sum_k \varphi_k
(E) \|B_k(\omega) - A(E,\omega)\|<  \sum_k  \varphi_k(E)\delta  =
\delta
$$
for all $E\in J$. Hence  by our choice of $\delta$ we infer
$$
\det \widetilde{A}(E,\omega) > 0 \quad \text{and} \quad \left\| \frac{1}{\sqrt{\det \widetilde{A}(E,\omega)}} \widetilde{A}(E,\omega) - \widetilde{A}(E,\omega) \right\| \leq \frac{\varepsilon}{2}
$$
for all $E\in J$ and $\omega\in\Omega$.  Define $A'$ on $J\times
\Omega$ by
$$
A'(E,\omega) := \frac{1}{\sqrt{\det \widetilde{A}(E,\omega)}} \widetilde{A}(E,\omega).
$$
Then, $A'$ is continuous with values in $\SL(2,\R)$ and $A'(E,\cdot)$ is locally constant (as the determinant of the locally constant $\widetilde{A}(E,\cdot)$ is locally constant). By construction we find
\begin{eqnarray*}
\|A'(E,\omega) - A(E,\omega)\| & \leq & \|A'(E,\omega) - \widetilde{A}(E,\omega)\| + \|\widetilde{A}(E,\omega) - A(E,\omega)\|\\
&\leq &  \frac{\varepsilon}{2} + \delta\\
& \leq& \varepsilon
\end{eqnarray*}
and the proof is finished.
\end{proof}

\begin{remark}
The proof carries over directly to  any compact topological space
$I$.
\end{remark}

From the preceding lemma and our main results we immediately obtain the following corollary.

\begin{coro}\label{c:abstract-general}
Let $(\Omega,T)$ be a subshift over a finite alphabet.

{\rm (a)}  If all locally constant cocycles on $\Omega$  have uniform
behaviour, then  for any interval $I$ the set $U(I,\Omega)$ is a
dense $G_\delta$-set.

{\rm (b)}  If all locally constant cocycles on $\Omega$ are uniform, then
for any interval $I$ the set $U(I,\Omega)$ is a dense
$G_\delta$-set.
\end{coro}

\begin{proof}
(a)  This follows from the preceding lemma and
Theorem~\ref{t:general-general}.

(b) This follows from the preceding lemma and Theorem~\ref{t:main-abstract}.
\end{proof}

\section{Generic Absence of Non-Uniform Hyperbolicity for Schr\"odinger Operators Over Boshernitzan Subshifts}\label{s.5}

In this section we show that for a generic continuous sampling function over an aperiodic subshift satisfying the Boshernitzan condition, the associated Schr\"odinger cocycles are uniform for all energies and the associated spectrum is a Cantor set of Lebesgue measure zero equal to the vanishing set of the Lyapunov exponent. That is, we prove Theorem~\ref{t.main} (and its corollary). We then
also point out a generalization.



Our proof of Theorem~\ref{t.main} relies on what we have shown in earlier sections  together with the  the following crucial feature of subshift satisfying (B).

\begin{lemma}[\cite{DL06a,DL06b}]\label{l:input-omega}
Let $(\Omega,T)$ be a subshift satisfying the Boshernitzan condition {\rm (B)}. Then any locally constant cocycle is uniform. In particular, if $(\Omega,T)$ is additionally assumed to be aperiodic, then  $\Sigma =\mathcal{Z}$ is a Cantor set of
Lebesgue measure zero for each Schr\"odinger operator associated to a locally constant $f\in C(\Omega,\R)$.
\end{lemma}

\begin{proof}[Proof of Theorem~\ref{t.main}.]
(a)   Clearly, the map
$$
S : C(\Omega) \longrightarrow C(\R \times \Omega, \SL(2,\R)), \; f
\mapsto \left( (E,\omega) \mapsto A^{E - f}(\omega) \right)
$$
is continuous.  Hence, the inverse image under $S$ of any $G_\delta$-set in $C(\R \times \Omega, \SL(2,\R))$ is a $G_\delta$-set in $C(\Omega)$. Thus, the set $\mathcal{G}$ consisting of $f \in C(\Omega)$ with $A(E, \cdot) := A^{E -
f(\cdot)} \in U(\R, \Omega)$ is a $G_\delta$-set by
Theorem~\ref{t:main-abstract}.

Moreover, for subshifts satisfying (B), it is known that any locally constant $f \in C(\Omega)$ yields a one-parameter family $A(E, \cdot) := A^{E - f(\cdot)} \in U(\R,\Omega)$; see Lemma~\ref{l:input-omega}. As locally constant $f\in C(\Omega)$ are
dense in $C(\Omega)$, we infer that the set $\mathcal{G}$ is dense as well. Altogether, this shows that $\mathcal{G}$ is a dense $G_\delta$-set.

Finally, as mentioned already, any subshift satisfying (B) is
uniquely ergodic and minimal. Hence, by the discussion in the
introduction, and in particular, by \eqref{e.Zero-uniform}, the
Schr\"odinger operator associated to $f\in C(\Omega,\R)$ satisfies
$\Sigma = \mathcal{Z}$ if and only if $\mathcal{NUH} =\emptyset$
holds, and this is the case if and only if the associated
Schr\"odinger cycle is uniform for all $E\in\R$, i.e. if and only if
$f\in \mathcal{G}$. As $\mathcal{G}$ is a $G_\delta$-set, this
proves (a).

\smallskip

(b) By Corollary~\ref{c.gdeltaset}, the set  $\{ f \in C(\Omega,\R)
: \Leb(\Sigma_f) = 0 \}$ is a $G_\delta$-set. Moreover, by
aperiodicity and (B)  this set  is dense by
Lemma~\ref{l:input-omega}. This shows (b).
\end{proof}

As a by-product of the considerations in the preceding proof we now deal with our result concerning the question of Walters.

\begin{proof}[Proof of Corollary~\ref{c:walters}.]
Any  locally constant cocycle on a subshift satisfying (B) is uniform, see Lemma~\ref{l:input-omega}. Now, the corollary is immediate from (b) of Corollary~\ref{c:abstract-general} (applied with an interval $I$ consisting of one point).
\end{proof}

Invoking \cite{AD05} we can give  also a variant of
Theorem~\ref{t.main}. This variant deals with  a more general
setting. We formulate it mainly as a reference point for potential
future generalizations.

\begin{theorem}\label{t.main2}
Let $(\Omega,T)$ be an aperiodic dynamical system. Assume that the set
$$
\{ f \in C(\Omega,\R) : A^{E-f} \mbox{ has uniform behavior for all } E \in \R \}
$$
is dense in $C(\Omega,\R)$. Then, for any ergodic measure $\mu$ on
$\Omega$, the set of $f \in C(\Omega,\R)$ for which we have that
$\mathcal{NUH} = \emptyset$ {\rm (}and hence $\Sigma = \mathcal{Z}${\rm )} and
$\Sigma$ is a Cantor set of Lebesgue measure zero is residual {\rm (}i.e., it contains a dense $G_\delta$-set{\rm )}.
\end{theorem}

\begin{proof}
Clearly, the map
$$
S : C(\Omega) \longrightarrow C(\R \times \Omega, \SL(2,\R)), \; f \mapsto \left( (E,\omega) \mapsto A^{E - f}(\omega) \right)
$$
is continuous.

Hence, the inverse image under $S$ of any $G_\delta$-set in $C(\R \times \Omega, \SL(2,\R))$ is a $G_\delta$-set in $C(\Omega)$. Thus, the set $\mathcal{G}$ consisting of $f \in C(\Omega)$ with $A(E, \cdot) := A^{E - f(\cdot)} \in W(\R, \Omega)$ is a $G_\delta$-set by Theorem~\ref{t:general-general}. Moreover, by assumption $\mathcal{G}$ is dense. Hence, $\mathcal{G}$ is a dense $G_\delta$-set and for each $f \in \mathcal{G}$, we have $\mathcal{NUH} = \emptyset$.  By \eqref{e.JL}, for all $f \in \mathcal{G}$ we then have $\Sigma = \mathcal{Z}$. Thus, the set of of $f$ such that $\mathcal{NUH} = \emptyset$ and $\Sigma =\mathcal{Z}$ holds is
residual.

Moreover, by \cite{AD05} and our aperiodicity assumption, the set of $f$ with $\Leb(\mathcal{Z}) = 0$ is a dense $G_\delta$-set and hence residual.

Since the intersection of two residual sets is residual and
$\Leb(\Sigma) = 0$ implies that $\Sigma$ is a Cantor set by general
principles (cf.~Remark~\ref{r.afterCorollary1.7}.(a)), we are done.
\end{proof}

\begin{remark}\label{r:method}
(a) Our proof of Theorem~\ref{t.main} works for  all uniquely ergodic minimal  subshifts for which locally constant cocycles are uniform as for these subshifts the conclusions of Lemma~\ref{l:input-omega} hold by \cite{DL06a}. Recent results show
the uniformity of locally constant cocycles for all simple Toeplitz subshifts \cite{GLNS, S19} (see \cite{LQ11} for related earlier work as well). All simple Toeplitz subshifts are minimal and uniquely ergodic.  Hence, the statement of Theorem~\ref{t.main} will hold for these subshifts as well. Note that the class of simple Toeplitz subshifts contains examples not satisfying the Boshernitzan
condition. A characterization of those simple Toeplitz subshifts satisfying the Boshernitzan condition is contained in \cite{LQ11}.
\\[1mm]
(b) Theorem~\ref{t.main2} does not require the dynamical system to be a subshift, nor does it require unique ergodicity or minimality. It can be applied to general ergodic dynamical systems. However, so far, the necessary denseness condition is only known for classes of uniquely ergodic minimal subshifts.
\\[1mm]
(c) Theorem~\ref{t.main2}  gives a slightly weaker conclusion than
Theorem~\ref{t.main} in that the involved sets are shown to be
residual rather then dense $G_\delta$-sets. The reason is that in
the  first part of the proof we obtain  the implication $f \in
\mathcal{G}\Longrightarrow \mathcal{NUH} =\emptyset$ but do not know
the converse (as we are dealing with general dynamical systems). If
additionally the condition of minimality and unique ergodicity is
imposed on the dynamical system, the converse holds and we can
conclude that the sets in question are $G_\delta$-sets (compare the
proof of Theorem~\ref{t.main} as well).
\\[1mm]
(d) The corresponding results hold for Jacobi operators. The alert reader may point out that in this more general setting, the standard transfer matrices are not given by $\SL (2,\R)$ cocycles, but rather by $\GL(2,\R)$ cocycles. However, it is not difficult (see, e.g., \cite{BP13,DKS10}) to identify an affiliated family of
$\SL (2,\R)$ cocycles whose study via the results above yields the desired conclusions.
\\[1mm]
(e) A similar remark applies in the setting of CMV matrices with dynamically defined Verblunsky coefficients. The necessary tools to adapt the present work to that setting are discussed in \cite{DL07}. The CMV analog of \cite{AD05}, as well as the adaptation of Corollary~\ref{c.main}, have been worked out in \cite{DFG}.
\end{remark}

\begin{remark} We  note that our proof of Theorem \ref{t:general-general} allows
for a (semi-) explicit construction of a potential with infinitely
many values (and arbitrarily close to any given potential) whose
cocycles are uniform for all energies whenever the underlying
dynamical system is a subshift $(\Omega,T)$ satisfying (B). The
point of the construction is that any finite sum of locally constant
functions $f : \Omega \to \R$ is locally constant again. Here are
the details:

Let $g_0$ be a locally constant function. Let $I$ be a compact interval containing an open neighborhood of the range of $g_0$. Let $g_n$, $n\in\N$, be an arbitrary
sequence of locally constant functions on $\Omega$ with $\|g_n\|
=1$ for each $n$. Let $\varepsilon_n\to 0$.

We now use the $W_\varepsilon$ from the proof of Theorem~\ref{t:general-general}. Consider $g_0$. Clearly $g_0$ belongs to $W_{\varepsilon_1}$ (as $g_0$ is locally constant).  As $W_{\varepsilon_1}$ is open, there exists a $\delta_1 > 0$ such that
any perturbation  of $g_0$ with sup norm not exceeding $\delta_1$ belongs to $W_{\varepsilon_1}$ as well. Without loss of generality, $\delta_1 < 1$.  Consider $g_0 + \frac{\delta_1}{2} g_1$. Clearly, this belongs to $W_{\varepsilon_2}$ (as it is locally constant). As $W_{\varepsilon_2}$ is open, there exists a $\delta_2>0$ such that any perturbation of $g_0 + \frac{\delta_1}{ 2} g_1$ with sup norm not exceeding $\delta_2$ belongs to $W_{\varepsilon_2}$. Without loss of generality $\delta_2 < \delta_1 / 2$. Inductively, we can then construct for each $N\in\N$ a  $\delta_N$ with $\delta_{N+1} < \delta_{N}/2$ such that any perturbation of $g_0 +
\frac{1}{2}\left(\sum_{j=1}^N \delta_j g_j\right)$ with sup norm not exceeding $\delta_{N+1}$ belongs to $W_{\varepsilon_{N+1}}$. Consider
$$
g:=g_0 + \lim_{N \to \infty} \Big( g_1 + \frac {1}{2}\sum_{j=1}^N \delta_j g_{j+1} \Big).
$$
By construction $g$ belongs to $W_{\varepsilon_j}$ for any $j\in\N$.
Now, the intersection of the $W_{\varepsilon_j}$ is $W(I,\Omega)$
(by definition of $W_\varepsilon$). This easily gives the desired
statement. As our choice of $g_0$ is arbitrary and $\sum \delta_j$
can be made arbitrarily small by making $\delta_1$ as small as
necessary, the function $g$ can be made arbitrarily close to any
given continuous function on $\Omega$.
\end{remark}

\begin{appendix}

\section{Notions of Uniform Hyperbolicity}\label{a.UH}

In this section we discuss various notions of uniform hyperbolicity in the context of continuous $\SL(2,\R)$ cocycles and the relationships between them. Related discussions can be found in \cite{DFLY16, V14, Y04, Z19}.

\medskip

Let $(\Omega,T)$ be a dynamical system. Denote the projective space over $\R^2$  consisting of lines through the origin by $\R \mathbb{P}^1$. This is a topological space in a natural way. Then any $B\in \SL (2,\R)$ can be considered as a map on $\R \mathbb{P}^1$ as it maps lines through the origin to lines through the origin. This map on $\R \mathbb{P}^1$ will be denoted by $B$ as well.

Let us consider the following three conditions for a continuous cocycle $A : \Omega\longrightarrow \SL(2,\R)$:

\begin{itemize}

\item[$\mathrm{(UH1)}$] There exists $L > 0$ with $\liminf_{n \to \infty} \frac{1}{n} \log \|A_n (\omega)\| \geq L$ uniformly in $\omega \in \Omega$.

\item[$\mathrm{(UH2)}$] There exists continuous maps $u, s : \Omega \longrightarrow \R \mathbb{P}^1$ as well as $\lambda > 1$ and $C > 0$ with

\begin{itemize}

\item[$(\alpha)$] $A (\omega) u(\omega) = u(T\omega) \mbox{ and } A(\omega) s(\omega) = s(T\omega)$ for all $\omega \in \Omega$;

\item[$(\beta)$] $\| A_n (\omega) U \|, \|A_{-n} S(\omega)\| \leq C \lambda^{-n}$ for all $n \in \N$, $\omega \in \Omega$ whenever $U \in u(\omega)$ and $S \in s(\omega)$ are normalized.

\end{itemize}

\item[$\mathrm{(UH3)}$] There exists $L > 0$ with $\lim_{n \to \infty} \frac{1}{n} \log \|A_n (\omega)\| = L$ uniformly in $\omega \in \Omega$.

\end{itemize}

\begin{prop}\label{p.uhimplications}
\begin{itemize}

\item[(a)] The conditions $\mathrm{(UH1)}$ and $\mathrm{(UH2)}$ are equivalent.

\item[(b)] $\mathrm{(UH3)}$ implies $\mathrm{(UH1)}$.

\item[(c)] If $(\Omega,T)$ is uniquely ergodic, then $\mathrm{(UH1)}$ is equivalent to  $\mathrm{(UH3)}$.

\end{itemize}
\end{prop}

\begin{proof}
(a) This is well-known; see, for example, \cite[Theorem~1.2]{DFLY16}, \cite[Proposition~2.5]{V14}, \cite[Proposition~2]{Y04}, and \cite[Corollary~1]{Z19}.

\smallskip

(b) This is obvious.

\smallskip

(c) By (a) and (b) it suffices to show $\mathrm{(UH2)} / \mathrm{(UH1)} \Longrightarrow \mathrm{(UH3)}$. This follows by standard methods as discussed, for example, in \cite{F97, L04}. More specifically, \cite[Theorem~3]{L04} shows that $\mathrm{(UH3)}$ follows from $\mathrm{(UH1)}$ under an additional minimality assumption. This minimality assumption is only used in the proof to ensure $(\beta)$ of $\mathrm{(UH2)}$. Hence,  the proof carries over
to our case.
\end{proof}

\begin{remark}\label{r.example}
It is not hard to see that the implication $\mathrm{(UH1)} \Longrightarrow \mathrm{(UH3)}$ fails whenever the system is not uniquely ergodic. Indeed consider a non-uniquely ergodic dynamical system $(\Omega,T)$. Then, there exists a continuous $f : \Omega \longrightarrow \R$ such that
$$
\frac{1}{n} \sum_{k=0}^{n-1} f(T^k \omega)
$$
does not converge uniformly in $\omega \in \Omega$. Without loss of generality we can assume $f \geq 1$ (otherwise replace $f$ by $f + 1 + \|f\|_\infty$). Set $h := \exp (f)$ and let $A : \Omega \longrightarrow \SL(2,\R)$ be given by
$$
A(\omega) = \begin{pmatrix} h(\omega) & 0 \\ 0 & 1/h(\omega) \end{pmatrix}.
$$
As $f \geq 1$, the cocycle $A$ clearly satisfies $\mathrm{(UH1)}$ with $L = 1$.  However, we have
$$
\frac{1}{n} \log \|A_n(\omega)\| =\frac{1}{n} \sum_{k=0}^{n-1} f(T^k \omega),
$$
which does not converge uniformly, and therefore $\mathrm{(UH3)}$ fails.
\end{remark}

\begin{coro}
Let $(\Omega,T)$ be uniquely ergodic. Then, an $A \in C(\Omega,\SL(2,\R))$ is uniform if and only if it has uniform behavior.
\end{coro}

\section{A Consequence of the Avalanche Principle}

The Avalanche Principle deals with products of $\SL(2,\R)$ matrices $A_N \ldots A_1$. Roughly stated, it asserts that the norm of this product is large once the norm of each $A_j$ and of the products $A_{j+1} A_{j}$ of consecutive matrices are large. It was introduced by Goldstein-Schlag in \cite{GS01} and then extended by Bourgain-Jitomirskaya in \cite{BJ02}. Subsequently, various further variations and extensions have been found; see, for example, \cite{B05, DK14, DK16, S13}. For us, the following consequence, essentially taken from \cite{DL06a} and based on \cite{BJ02}, will be relevant.

\begin{lemma}\label{l:consequence-avalanche}
There exist constants $\kappa > 0$ and $\lambda_0 > 0$ such that the
following holds. Let $(\Omega,T)$ be a dynamical system and $A :
\Omega \longrightarrow \SL(2,\R)$. Let $0 < \varepsilon <
\frac{1}{12}$ be arbitrary. Assume that there exist $\ell \in \N$
and $L
>0$ with
\begin{itemize}

\item[(a1)] $\frac{1}{n} \log\|A_n (\omega)\|\leq L(1 + \varepsilon)$ for all $\omega \in \Omega$ and $n\geq l$,

\item[(a2)] $L (1 - \varepsilon) \leq \frac{1}{2l}\log \|A_{2l} (\omega)\|$ for all $\omega \in \Omega$,

\item[(a3)] $ \frac{3}{4} L \ell \geq \lambda_0$,

\item[(a4)] $\frac{1}{\ell} \frac{2 \kappa}{ \exp(\lambda_0)} < \varepsilon L$.

\end{itemize}

Then,
$$
L (1- 44\varepsilon) \leq \frac{1}{n} \log\|A_n (\omega)\| \leq L (1 + \varepsilon)
$$
for all $\omega \in \Omega$ and $n \geq \ell$.
\end{lemma}

\begin{proof}
The assumptions (a1), (a2), (a3) and (a4) of the lemma are just the conditions (I), (II), (III), (IV) appearing in the proof of \cite[Theorem~1]{DL06a}. The lower bound given in the conclusion of the lemma then follows by following this proof verbatim. The upper bound is obvious from the assumptions.
\end{proof}

\begin{remark} \label{r:nach-avalanche-lemma}
(a) Let us emphasize that the number $L$ appearing in the lemma is not required to be the Lyapunov exponent of $A$. It suffices that it is sufficiently close to the actual Lyapunov exponent. This is
relevant for an application to families of cocycles.
\\[1mm]
(b) It may be instructive to discuss the assumptions appearing in the lemma: The assumptions (a3) and (a4) are independent of $A$. For given $L > 0$ and $\varepsilon > 0$, they will be satisfied for all large enough $\ell$. For uniquely ergodic systems, the assumption (a1) is automatically satisfied for any given $\varepsilon$ if $L = L(A)$ and $\ell$ is large enough. So, in this sense for uniquely ergodic dynamical systems, the crucial condition is the second assumption (a2).
\\[1mm]
(c) Note  that the assumptions of the lemma are open conditions in the following sense: Consider an $A$ satisfying the assumptions for $\ell \in \N$, $L > 0$ and $\varepsilon > 0$.  Now, let $\varepsilon' > \varepsilon$ (with $\varepsilon' < 1/12$) be given. Then, any $B$ sufficiently close to $A$ will satisfy the assumptions of the lemma with the same $\ell$ and $L$ and $\varepsilon$ replaced by $\varepsilon'$. Indeed, the last two assumptions (a3) and (a4) do not depend on $A$ and are then clearly satisfied for $B$. The second assumption (a2) is satisfied for $B$ sufficiently close to $A$ due to $\varepsilon' > \varepsilon$. Similarly, the first assumption (a1) is satisfied for $B$ sufficiently close to $A$ by part (a) of Lemma~\ref{l:aux}.
\end{remark}

\end{appendix}

\end{document}